\documentclass[12pt]{article}
\usepackage[a4paper, mag=1000, includefoot,
            left=3cm, right=1cm, top=2cm, bottom=2cm, headsep=1cm, footskip=1cm]{geometry}
\usepackage[english]{babel}
\usepackage{footnpag}
\usepackage{amssymb} 
\usepackage{amsmath}
\usepackage{amsthm}
\usepackage{eucal}
\usepackage{dsfont}
\usepackage{graphicx}
\usepackage{psfrag}
\usepackage{caption}
\usepackage[numbers,square,comma,compress]{natbib}
\usepackage[unicode,hypertex,final,pagebackref,colorlinks,bookmarksnumbered]{hyperref}
\usepackage{ifpdf}
\ifpdf\usepackage{bookmark}\fi 
\allowdisplaybreaks
\newtheorem{lem}{Lemma}
\newtheorem{pro}{Proposition}
\theoremstyle{definition}
\newtheorem{Def}{Definition}
\theoremstyle{remark}
\newtheorem{rem}{Remark}[section]

\DeclareMathOperator{\re}{Re}

\newcommand{\txs}{\textstyle}

\newcommand{\lt}{\left}
\newcommand{\rt}{\right}

\newcommand{\ts}{\times}

\renewcommand{\le}{\leqslant}
\renewcommand{\ge}{\geqslant}
\newcommand{\hm}[1]{#1 \nolinebreak \discretionary{}{\mbox{$#1$}}{}}
\newcommand{\intl}[2]{\int\limits_{#1}^{#2}}
\newcommand{\tintl}[2]{{\txs \intl{#1}{#2} }\,}

\newcommand{\zt}{\zeta}

\newcommand{\vkp}{\varkappa}

\newcommand{\la}{\lambda}
\newcommand{\La}{\Lambda}

\newcommand{\tl}{\tilde}

\newcommand{\ol}{\overline}

\newcommand{\RR}{\mathds{R}}
\newcommand{\NN}{\mathds{N}}

\newcommand{\CC}{\mathds{C}}

\hypersetup{
  pdftitle = Uniform exponential-power estimate for the solution to a family of the Cauchy problems for linear differential equations,
  pdfauthor = {Evgeny E. Bukzhalev},
  pdfsubject = {ordinary differential equations},
  pdfkeywords = {families of linear differential equations, estimates for solutions of differential equations, initial value problem for ordinary differential equation, estimates for roots of polynomials, Routh\unichar{"2013}Hurwitz stability criterion}}

\newcommand{\bLm}{\bar\Lambda^m}

\begin{document}

\title{Uniform exponential-power estimate\\ for the solution to a family of the Cauchy problems\\ for linear differential equations}

\author{Evgeny E. Bukzhalev\thanks{E-mail: bukzhalev@mail.ru}\\ \em \small M.~V.~Lomonosov Moscow State University, Moscow, Russia} 

\maketitle

\begin{abstract}

We consider a solution to a parametric family of the Cauchy problems for $m$th-order
linear differential equations with constant coefficients. Parameters of the family are the coefficients of the differential equation and the initial values of the solution and its derivatives up to the $(m-1)$th-order (by a~solution to a family of problems we mean a function of the parameters of the given family that maps each tuple of parameters to a solution to the problem with these parameters). We obtain an exponential-power estimate for the functions of this parametric family that is uniform (with respect to parameters) on any bounded set. We also prove that the maximal element of the set of real parts of monic polynomial roots is a~continuous function (of the coefficients of the polynomial). The continuity of this element is used for obtaining the estimate mentioned above (since to each tuple of coefficients of the differential equation there corresponds its characteristic polynomial with these coefficients, the set of the roots of the characteristic polynomial and the maximal element of this set are also functions of the coefficients of the differential equation).


\end{abstract}

\section{Introduction}

We obtain a uniform (with respect to parameters $M_m \in \CC^m$ and $N_m \in \CC^m$) exponential-power estimate for the solution $w(\cdot; M_m, N_m): [0, +\infty) \to \CC$ to the Cauchy problem for a linear differential equation of an arbitrary order $m \in \NN$ with constant coefficients $M_m$ and initial values $N_m$ considered as parameters for $w$:
\begin{gather}\label{de wn}
  w^{(m)}(\xi; M_m, N_m) = a_{m-1}\, w^{(m-1)}(\xi; M_m, N_m) + \cdots + a_0\, w(\xi; M_m, N_m), \quad \xi \in (0,+\infty);
\\\label{ic wn}
  w(0;M_m,N_m) = w^0,\ \ldots,\ w^{(m-1)}(0;M_m,N_m) = w^{m-1},
\end{gather}
where $M_m = (a_0, \dots, a_{m-1}) \in \CC^m$, $N_m = (w^0, \dots, w^{m-1}) \in \CC^m$, $w^{(i)}$ is the $i$th derivative of the function $w(\cdot; M_m, N_m)$ (i.e., the $i$th derivative of the function $w$ with respect to the first argument). We also estimate the first $(m-1)$ derivatives of the function $w(\cdot; M_m, N_m)$ (recall that the Cauchy problem for equation \eqref{de wn} has a unique solution for any $a_i$ and $w^i$). But
 since in view of equation \eqref{de wn} the derivatives of $w(\cdot; M_m, N_m)$
 of order $n\ge m$  is a~linear combination of lower-order derivatives,
it follows that in fact the uniform exponential-power estimates hold for an arbitrary-order derivative of $w(\cdot; M_m, N_m)$.

\begin{Def}
Let $\mathfrak K$ be the map that to each double $(M_m,N_m) \in {(\CC^m)}^2 \hm=: D_{\mathfrak K}$ assigns
problem \eqref{de wn}--\eqref{ic wn}. The map $\mathfrak K$ is called the family of the Cauchy problems for the $m$th-order differential equation \textup(with constant coefficiants\textup), and the components $M_m$ and $N_m$ of double $(M_m,N_m) \in D_{\mathfrak K}$ are called the parameters of the family $\mathfrak K$.
\end{Def}

\begin{rem}
  The restriction of the map $\mathfrak K$ to a set $\mathcal M \subseteq D_{\mathfrak K}$ is often called subfamily of the family $\mathfrak K$ and is denoted by $\{ \mathfrak K(M_m, N_m) \}_{(M_m,N_m) \in \mathcal M}$.
\end{rem}

\begin{Def}
  Let $W$ be the map that takes each double $(M_m,N_m) \in D_{\mathfrak K}$ to the solution to the problem $\mathfrak K(M_m,N_m)$ (i.g., the solution to problem \eqref{de wn}--\eqref{ic wn}). The map $W$ is said to be the solution to the family $\mathfrak K$.
\end{Def}

\begin{rem}
  The value $W(M_m,N_m)$ of the map $W$ is a function of one real variable. This function is the solution to problem \eqref{de wn}--\eqref{ic wn}. And since we already denoted the solution to this problem by $w(\cdot; M_m, N_m)$, there is the following relation between $W$ and $w$: $W(M_m,N_m)(\xi) = w(\xi; M_m, N_m)$ (strictly speaking, just this relation defines $w$).
\end{rem}

Let $\mathfrak E$ be the map that to each single $M_m \in \CC^m$ assigns the differential equation of the family $\mathfrak K(M_m,\cdot) := \{ \mathfrak K(M,N_m) \}_{(M,N_m) \in \{M_m\} \ts \CC^m}$ (see \eqref{de wn}). It is known (see, e.g., \cite{Tikhonov_Vasil'eva_Sveshnikov_1985_book}) that the structure of the solution $W(M_m,N_m)$ to the problem $\mathfrak K(M_m,N_m)$ depends on the distribution of the multiplicity of the roots of the characteristic polynomial $\mathfrak P(M_m)$ (see \eqref{PM}) for the equation $\mathfrak E(M_m)$ (here $M_m$ is a tuple of coefficients $a_i$ of both the equation and its characteristic polynomial). Therefore, the explicit formula for $w^{(i)}(\xi; M_m, N_m)$ can hardly be used to obtain uniform estimates for the functions $w^{(i)}(\cdot; M_m, N_m)$ with respect to  $M_m$ and $N_m$ (the exceptions are the cases of $m=1$ and $m=2$ (see \cite{Bukzhalev_2017_CMMP, Bukzhalev_2017_MUCMC}, where these estimates are obtained for the case of a~second-order differential equation ($m=2$) with constant real coefficients)).

In the sequel we need one uniform (with respect to the coefficients $a_i$) estimate for $\bLm$,
which is the greatest of the real parts of the roots of the characteristic polynomial for the equation $\mathfrak E(M_m)$.
Its derivation is based on the continuity of $\bLm$ (as a function of $M_m$) and we begin our paper with a proof of this continuity.

\section{Proof of the continuity of \boldmath \texorpdfstring{$\bLm$}
{\unichar{"035E}\unichar{"1D27}\unichar{"1D50}}}

Let $\mathfrak P$ be the map that to each $M_m = (a_0, \dots, a_{m-1}) \in \CC^m$ assigns the characteristic polynomial for the equation $\mathfrak E(M_m)$ (see \eqref{de wn}):
\begin{gather}\label{PM}
  \mathfrak P(M_m) := \la^m - a_{m-1}\, \la^{m-1} - \cdots - a_1\, \la - a_0.
\end{gather}
The map $P$ is also called a~family of polynomials. Since the degree of the polynomial $\mathfrak P(M_m)$ is $m$ for all $M_m \in \CC^m$, there exist functions $\la^1$, \dots, $\la^m$: $\CC^m \to \CC$ such that
\begin{gather}\label{PME}
  \forall M_m \in \CC^m\ \mathfrak P(M_m) = (\la - \la^1(M_m)) \ldots (\la - \la^m(M_m)).
\end{gather}
The numbers $\la^1(M_m)$, \dots, $\la^m(M_m)$ are called roots of the polynomial $\mathfrak P(M_m)$. A tuple $(\la^1, \dots,$ $\la^m)$ of the functions $\la^i$: $\CC^m \to \CC$ that satisfy condition \eqref{PME} is called a full tuple of roots of the family $\mathfrak P$. Further, by the continuity (on set $\mathds M$) of a tuple $(\la^1, \dots, \la^m)$ we mean the continuity (on $\mathds M$) of each its component.

Denote by $\mathcal L$ the set of all maps $\CC^m \to \CC^m$, $M \mapsto (\la^1(M), \dots, \la^m(M))$,
such that $(\la^1, \dots,$ $\la^m)$ is a~full tuple of roots of the family $\mathfrak P$.
There are infinitely many full tuples of roots, since
for each $M_m \in \CC^m$ there are various ways to label the roots of the polynomial $\mathfrak P(M_m)$
(we assume that each root is repeated as many times as its multiplicity).
It is easy to verify that for $m \ge 2$ the set $\mathcal L$ contains no maps continuous in the whole space~$\CC^m$. Moreover, for $m \ge 4$ the set $\mathcal L$ contains no maps with continuous restriction to $\RR^m$ (the proof can be found in \cite{2017arXiv171000640B}). However, it is known that for each $m$ and any point~$M_0 \in \CC^m$, there exists a map $\La_{M_0} \in \mathcal L$ continuous at this point (see, e.g., \cite{Ostrowski_1966_book}).

\begin{rem}
  Each map $\La_{M_0}$ is continuous at the corresponding point $M_0$. I.g., first we fix point $M_0$ and then choose
  the map $\La_{M_0}$ continuous at this point (the map $\La_{M_0}$ can be discontinuous at other points).
  If we change the point then, generally speaking, we will have to change the map. Each point $M_0$ has its nonempty set $\mathfrak S_{M_0}$ of full tuples of roots of the family $\mathfrak P$ continuous at this point and the intersection of the family of sets $\{ \mathfrak S_{M_0} \}_{M_0 \in \CC^m}$ for $m \ge 2$ is empty.
\end{rem}

\begin{lem}\label{lem}
  Let $\mathcal L \ni \La: M \mapsto (\la^1(M), \dots, \la^m(M))$. Then
  \begin{gather}\label{La}
    \bLm: \CC^m \to \RR,\ M_m \mapsto \max \{\re\la^1(M_m), \dots, \re\la^m(M_m)\}
  \end{gather}
  is a continuous function on $\CC^m$.
\end{lem}

\begin{rem}
For each point $M_m \in \CC^m$ the unordered set of roots of the polynomial $\mathfrak P(M_m)$ and the value $\bLm(M_m)$ are independent of the choice of
$\La\in \mathcal L$. Thus, to each $\La \in \mathcal L$ (i.e., to each way of numbering of roots of the family $\mathfrak P$) there corresponds
the same function~$\bLm$.
\end{rem}

\begin{proof}[Proof of Lemma \ref{lem}]
We fix an arbitrary point $M_0 \in \CC^m$ and choose a map $\La_{M_0} \in \mathcal L$ which is
continuous at this point. Let $\La_{M_0}: M \mapsto (\la^1_{M_0}(M), \dots, \la^m_{M_0}(M))$. Then each of the functions $\la^i_{M_0}$ is also continuous at the point $M_0$. But the continuity of $\la^i_{M_0}$ implies that of $\re \la^i_{M_0}$, whereas the continuity of all $\re \la^i_{M_0}$ implies the continuity of the maximum of these functions.
\end{proof}

\section{Obtaining exponential-power estimates}

We put $\Pi_m(C) := \{ (x_1, \dots, x_m) \in \CC^m : |x_1| \le C, \dots, |x_m| \le C \}$.

\begin{pro}\label{pro}

Let $C_a \ge 0$, $C_w \ge 0$. Then there exists $\tl C_m \ge 0$ such that
\begin{gather}\label{est w}
  \lt| w^{(i)}(\xi;M_m,N_m) \rt| \le \tl C_m\, (1 + \xi^{m-1})\, e^{\bLm(M_m)\, \xi}
\end{gather}
for all $(i,\xi,M_m,N_m) \in \{ 0, \dots, m-1 \} \ts [0, +\infty) \ts \Pi_m(C_a) \ts \Pi_m(C_w)$, where $w(\cdot; M_m, N_m)$ is the solution to problem \eqref{de wn}--\eqref{ic wn}, $\bLm$ is the function from Lemma \ref{lem}.

\end{pro}

\begin{proof}

We use induction on $m$. We denote by $S_m$ the conclusion of Proposition~\ref{pro}.
Since $S_1$ is clearly true, it remains to verify, that $S_{m+1}$ follows from $S_m$ implies for any $m \in \NN$.

Consider the Cauchy problem for the $(m+1)$th-order equation:
\begin{gather}\label{de wn+1}
  w^{(m+1)}(\xi; M_m, N_m) = a_m\, w^{(m)}(\xi; M_m, N_m) + \cdots + a_0\, w(\xi; M_m, N_m), \quad \xi \in (0,+\infty);
\\\label{ic wn+1}
  w(0;M_{m+1},N_{m+1}) = w^0,\ \ldots,\ w^{(m)}(0;M_{m+1},N_{m+1}) = w^m
\end{gather}
and fix arbitrary $C_a \ge 0$ and $C_w \ge 0$. The statement of $S_{m+1}$ is as follows: there exists $\tl C_{m+1}$ such that
\begin{gather*}
  \lt| w^{(i)}(\xi;M_{m+1},N_{m+1}) \rt| \le \tl C_{m+1}\, (1 + \xi^m)\, e^{\bar \La_{m+1}(M_{m+1})\, \xi}
\end{gather*}
for all $(i,\xi,M_{m+1},N_{m+1}) \in \ol{0,m} \ts [0, +\infty) \ts \Pi_{m+1}(C_a) \ts \Pi_{m+1}(C_w)$, where $w(\cdot; M_{m+1}, N_{m+1})$ is the solution to problem \eqref{de wn+1}--\eqref{ic wn+1}.

To verify the validity of $S_{m+1}$ (assuming that  $S_m$ is true) we change the variable in problem \eqref{de wn+1}--\eqref{ic wn+1}:
\begin{gather}\label{u}
  w(\xi;M_{m+1},N_{m+1}) = e^{\la^*(M_{m+1})\, \xi}\, u(\xi;M_{m+1},N_{m+1}),
\end{gather}
where $\la^*$ is a function that takes each tuple $M_{m+1} = (a_0, \dots, a_m) \in \CC^{m+1}$ to an arbitrary root $\la_i(M_{m+1})$ of the characteristic polynomial for equation \eqref{de wn+1} whose real part $\re \la_i(M_{m+1})$ coincides with $\bar \La_{m+1}(M_{m+1})$:
\begin{gather}\label{re la*}
  \re \la^*(M_{m+1}) = \bar \La_{m+1}(M_{m+1}).
\end{gather}

For the new function $u(\cdot; M_{m+1}, N_{m+1}): [0, +\infty) \to \CC$ we obtain the following initial-value problem:
\begin{gather}\notag
  u^{(m+1)}(\xi;M_{m+1},N_{m+1}) = b_m(M_{m+1})\, u^{(m)}(\xi;M_{m+1},N_{m+1}) + \cdots +{}
\\\label{de u}
  {}+ b_1(M_{m+1})\, u'(\xi;M_{m+1},N_{m+1}), \quad \xi \in (0,+\infty);
\\\label{ic u}
  u(0;M_{m+1},N_{m+1}) = u^0(M_{m+1},N_{m+1}),\ \ldots,\ u^{(m)}(0;M_{m+1},N_{m+1}) = u^m(M_{m+1},N_{m+1}).
\end{gather}
Here $b_i(M_{m+1}) = \tl b_i(\la^*(M_{m+1}), M_{m+1})$, $u^i(M_{m+1},N_{m+1}) = \tl u^i(\la^*(M_{m+1}), N_{m+1})$, where, in its turn, $M_{m+1} = (a_0, \dots, a_m)$ and $N_{m+1} = (w^0, \dots, w^m)$, $\tl b_i$ and $\tl u^i$ are the known
functions of $m+2$ arguments (polynomial with respect to the first argument and linear with respect to other $m+1$ arguments). In equation \eqref{de u}, we already took into account that its characteristic polynomial has the zero as a~root
for $M_{m+1} \in \CC^{m+1}$ (see \eqref{mu}), so the coefficient $b_0(M_{m+1})$ of the $u(\xi;M_{m+1},N_{m+1})$ vanishes identically.

Due to \eqref{u}, for each $M_{m+1} \in \CC^{m+1}$ the roots of the characteristic polynomial for equation \eqref{de u} are as follows:
\begin{gather}\label{mu}
  \mu_i(M_{m+1}) := \la_i(M_{m+1}) - \la^*(M_{m+1}), \quad i \in \{1, \dots, m+1\}.
\end{gather}
This and the definition of $\la^*(M_{m+1})$ yield
\begin{gather}\label{Re mu}
  \re \mu_i(M_{m+1}) \le 0
\end{gather}
for all $(i,M_{m+1}) \in \{1, \dots, m+1\} \ts \CC^{m+1}$.

Since we assume that the points $M_{m+1} = (a_0, \dots, a_m)$ lie in the finite parallelepiped $P_{m+1}(C_a)$, all roots $\la_i(M_{m+1})$ of the characteristic polynomial for equation \eqref{de wn+1} satisfy the inequality (see, e.g., \cite{Markushevich_2005_book}):
\begin{gather}\label{la}
  |\la_i(M_{m+1})| \le 1 + C_a.
\end{gather}
So there exist constants $C_b \ge 0$ and $C_u \ge 0$ such that
\begin{gather}\label{bi ui}
  |b_i(M_{m+1})| \le C_b, \quad |u^i(M_{m+1},N_{m+1})| \le C_u
\end{gather}
for all $(i,M_{m+1},N_{m+1}) \in \ol{0,m} \ts \Pi_{m+1}(C_a) \ts \Pi_{m+1}(C_w)$.

We reduce the order of equation \eqref{de u} by one more change of variable:
\begin{gather}\label{v}
  u'(\xi;M_{m+1},N_{m+1}) = v(\xi;M_{m+1},N_{m+1}).
\end{gather}

The function $v(\cdot; M_{m+1}, N_{m+1}): [0, +\infty) \to \CC$ satisfies the following initial-value problem:
\begin{gather}\notag
  v^{(m)}(\xi;M_{m+1},N_{m+1}) = b_m(M_{m+1})\, v^{(m-1)}(\xi;M_{m+1},N_{m+1}) + \cdots +{}
\\\label{de v}
  {}+ b_1(M_{m+1})\, v(\xi;M_{m+1},N_{m+1}), \quad \xi \in (0,+\infty);
\\\notag
  v(0;M_{m+1},N_{m+1}) = u^1(M_{m+1},N_{m+1}),\ \ldots,\ v^{(m-1)}(0;M_{m+1},N_{m+1}) = u^m(M_{m+1},N_{m+1}).
\end{gather}

Let $\nu_1(M_{m+1})$, \dots, $\nu_m(M_{m+1})$ be the roots of the characteristic polynomial for equation \eqref{de v}. Since each
$\nu_i(M_{m+1})$ is at the same time a root of the characteristic polynomial for equation \eqref{de u}, they satisfy the same inequality as $\mu_i(M_{m+1})$ (see \eqref{Re mu}) for all $M_{m+1} \in \CC^{m+1}$:
\begin{gather}\label{Re nu}
  \re \nu_i(M_{m+1}) \le 0.
\end{gather}

Note also that (see \eqref{bi ui})
\begin{gather*}
  M_m = (b_1(M_{m+1}), \dots, b_m(M_{m+1})) \in \Pi_m(C_b),
\\
  N_m = (u^1(M_{m+1},N_{m+1}), \dots, u^m(M_{m+1},N_{m+1})) \in \Pi_m(C_u)
\end{gather*}
for all $M_{m+1} \in \Pi_{m+1}(C_a)$ and $N_{m+1} \in \Pi_{m+1}(C_w)$. The last estimates allow one to apply the induction hypothesis to the function $v$: there exists $\tl C_m \ge 0$ such that (see \eqref{est w}, \eqref{La} and \eqref{Re nu})
\begin{gather}\label{est v}
  \lt| v^{(i)}(\xi;M_{m+1},N_{m+1}) \rt| \le \tl C_m\, (1 + \xi^{m-1})
\end{gather}
for all $(i,\xi,M_{m+1},N_{m+1}) \in \{ 0, \dots, m-1 \} \ts [0, +\infty) \ts \Pi_{m+1}(C_a) \ts \Pi_{m+1}(C_w)$.

From \eqref{v} and \eqref{est v} we immediately obtain
\begin{gather}\label{est du}
  \lt| u^{(i)}(\xi;M_{m+1},N_{m+1}) \rt| = \lt| v^{(i-1)}(\xi;M_{m+1},N_{m+1}) \rt| \le \tl C_m\, (1 + \xi^{m-1})
\end{gather}
for the first $m$ derivatives of the function $u(\cdot; M_{m+1}, N_{m+1})$
(here and until the end of the proof we assume that $(\xi,M_{m+1},N_{m+1}) \in [0, +\infty) \hm\ts \Pi_{m+1}(C_a) \ts \Pi_{m+1}(C_w)$).

To estimate the function $u$ itself, we integrate \eqref{v} and then use \eqref{ic u}, \eqref{bi ui}, and \eqref{est v} and
employ the monotonicity property and the estimate of the absolute value of the definite integral:
\begin{gather}\notag
  \Big| u(\xi;M_{m+1},N_{m+1}) \Big| = \Big| u(0;M_{m+1},N_{m+1}) + \tintl0\xi\, v(\zt;M_{m+1},N_{m+1})\, d\zt \Big| \le \Big| u^0(M_{m+1},N_{m+1}) \Big| +{}
\\\label{est u}
  {}+ \tintl0\xi\, \Big| v(\zt;M_{m+1},N_{m+1}) \Big|\, d\zt \le C_u + \tintl0\xi\, \tl C_m\, (1 + \xi^{m-1})\, d\zt \le \tl C_u\, (1 + \xi^m)
\end{gather}
for sufficiently large $\tl C_u$.

We turn back to $w$. From \eqref{u}, \eqref{est u}, \eqref{est du}, \eqref{la}, \eqref{re la*} and
the Leibniz formula (for the $i$th-order derivative of the product of two functions) for each $i \in \ol{0,m}$ and sufficiently large $\tl C_{m+1}$ we have
\begin{gather*}
  \lt| w^{(i)}(\xi;M_{m+1},N_{m+1}) \rt| \le \sum_{j=0}^i \frac{i!}{j!(i-j)!}\, \big| u^{(j)}(\xi;M_{m+1},N_{m+1}) \big|\, {\big| \la^*(M_{m+1}) \big|}^{i-j}\, \big| e^{\la^*(M_{m+1})\, \xi} \big| \le
\\
  \le \Big[ \tl C_u\, (1 + \xi^m)\, {(1 + C_a)}^i + \sum_{j=1}^i \frac{i!}{j!(i-j)!}\, \tl C_m\, (1 + \xi^{m-1})\, {(1 + C_a)}^{i-j} \Big]\, e^{\re \la^*(M_{m+1})\, \xi} \le
\\
  \le \tl C_{m+1}\, (1 + \xi^m)\, e^{\bar \La_{m+1}(M_{m+1})\, \xi}.
\end{gather*}

\end{proof}

\begin{rem}
  Since due to equation \eqref{de wn}, the derivative of the function $w(\cdot; M_m, N_m)$ of the order $i \ge m$ is a linear combination of its lower derivatives, by induction on $i$ one can easily verify that for any nonnegative integer $i$ there exists $\tl C^i_m \ge 0$ such that
  \begin{gather*}
    \lt| w^{(i)}(\xi;M_m,N_m) \rt| \le \tl C^i_m\, (1 + \xi^{m-1})\, e^{\bLm(M_m)\, \xi}
  \end{gather*}
  for all $(\xi,M_m,N_m) \in [0, +\infty) \ts \Pi_m(C_a) \ts \Pi_m(C_w)$.
\end{rem}

\begin{rem}
  If we replace the rectangles $\Pi_m(C_a)$ and $\Pi_m(C_w)$ by arbitrary bounded sets $\mathds D_a$ and $\mathds D_w$ (they can also be sets of $\RR^m$) in Proposition~\ref{pro}, then it obviously remain valid.
\end{rem}

\begin{rem}
  By the uniformity of estimate \eqref{est w} we mean the independence of the coefficient $\tl C_m$ of the parameters $a_i$ and $w^i$. At the same time the coefficient of $\xi$ in the argument of the exponential function depends on the parameters $a^i$. Thus, estimate \eqref{est w} is only semi-uniform
  in some sense. Of course we can apply \eqref{la} and replace $\bLm(M_m)$ by $1 + C_a$ in \eqref{est w}, that gives us completely uniform but, generally speaking, more rough estimate. However, such loss of accuracy is sometimes undesirable, especially if the coefficient of $\xi$ in the argument of
  the exponential function change its sign (from negative to positive)---this is the case that will be discussed below.
\end{rem}

Let $\mathds D_a$ be a closed bounded set and let it be known (e.g., due to Routh--Hurwitz criterion, see \cite{Gantmacher_2000_book}) that $\forall (i,M_m) \in  \ol{1,m} \ts \mathds D_a$ $\re\la^i(M_m) < 0$. Then $\forall M_m \in \mathds D_a$ $\bLm(M_m) < 0$, and since due to Lemma \ref{lem} the function $\bLm$ is continuous on the whole $\CC^m$ (and hence it is continuous on any set $\mathds D_a$ of this space),
it follows from Weierstrass's theorem on the maximum of a~continuous function
that $\exists M^*_m \in \mathds D_a$ $\forall M_m \in \mathds D_a$ $\bLm(M_m) \le \bLm(M^*_m) < 0$. This and Proposition~\ref{pro} imply the existence of $\tl C_m > 0$ such that for all $(i,\xi,M_m,N_m) \in \{ 0, \dots, m-1 \} \ts [0, +\infty) \ts \mathds D_a \ts \mathds D_w$ (here $\mathds D_w$ is an arbitrary bounded set of $\CC^m$) the solution $w(\cdot; M_m, N_m)$ of problem \eqref{de wn}--\eqref{ic wn} satisfies the inequality
\begin{gather}\label{est_w}
  \lt| w^{(i)}(\xi;M_m,N_m) \rt| \le \tl C_m\, (1 + \xi^{m-1})\, e^{- \varkappa\, \xi},
\end{gather}
where $\varkappa := - \bLm(M^*_m) > 0$.

Finally, we consider the family (with respect to parameters, which are listed below) of the Cauchy problems for linear differential equation of an arbitrary order $m \in \NN$ with coefficients depending on the parameters $t_1$, \dots, $t_k$ (moreover, the initial values of the solution $w(\cdot; M_k, N_m)$ and its derivatives still act as parameters)
\begin{gather}\notag
  w^{(m)}(\xi; M_k, N_m) = a_{m-1}(t_1, \dots, t_k)\, w^{(m-1)}(\xi; M_k, N_m) + \cdots +{}
\\\label{de_wnt}
  {}+ a_0(t_1, \dots, t_k)\, w(\xi; M_k, N_m) = 0, \quad \xi \in (0,+\infty);
\\\notag
  w(0;M_k,N_m) = w^0,\ \ldots,\ w^{(m-1)}(0;M_k,N_m) = w^{m-1},
\end{gather}
where $M_k = (t_1, \dots, t_k) \in \mathds D_t \subseteq \CC^k$, $N_m = (w^0, \dots, w^{m-1}) \in \CC^m$, $a_i: \mathds D_t \to \CC$.

\begin{pro}\label{sta_t}

Suppose $\mathds D_t$ is a closed bounded set of $\CC^k$, $\mathds D_w$ is a bounded set of $\CC^m$, $a: \mathds D_t \ni M_k \mapsto (a_0(M_k), \dots, a_{m-1}(M_k)) \in \CC^m$, the functions $a_0$, \dots, $a_{m-1}$ are continuous on $\mathds D_t$, $\forall M_k \in \mathds D_t$ $\bLm(a(M_k)) < 0$ \textup(here $\bLm$ is the function from Lemma~\ref{lem}\textup). Then there exists $\tl C_m \ge 0$ and $\vkp > 0$ such that
\begin{gather*}
  \lt| w^{(i)}(\xi;M_k,N_m) \rt| \le \tl C_m\, (1 + \xi^{m-1})\, e^{ - \vkp\, \xi}
\end{gather*}
for all $(i,\xi,M_k,N_m) \in \{ 0, \dots, m-1 \} \ts [0, +\infty) \ts \mathds D_t \ts \mathds D_w$, where $w(\cdot; M_k, N_m)$ is the solution to problem \eqref{de_wnt}.

\end{pro}

\begin{proof}
Since $a(\mathds D_t) =: \mathds D_a$ is a bounded closed set of $\CC^m$, to prove Proposition~\ref{sta_t} it suffices to note that
\[
  \max\limits_{M_k \in \mathds D_t} \bLm(a(M_k)) = \max\limits_{M_m \in \mathds D_a} \bLm(M_m),
\]
and then put $\vkp := - \max\limits_{M_k \in \mathds D_t} \bLm(a(M_k))$ and apply estimate \eqref{est_w}.
\end{proof}

\bibliography{Singul_Eng,Bukzhalev_Articles,Bukzhalev_Preprints}

\providecommand{\href}[2]{#2}\begingroup\raggedright\begin{thebibliography}{1}

\bibitem{Tikhonov_Vasil'eva_Sveshnikov_1985_book}
A.~N. Tikhonov, A.~B. Vasil'eva, and A.~G. Sveshnikov, {\em Differential
  Equations}.
\newblock Springer Series in Soviet Mathematics. Springer-Verlag, Berlin,
  Heidelberg, 1st~ed., 1985.

\bibitem{Bukzhalev_2017_CMMP}
E.~E. Bukzhalev, ``On one method for the analysis of the {C}auchy problem for a
  singularly perturbed inhomogeneous second-order linear differential
  equation,'' \href{http://dx.doi.org/10.1134/S0965542517100050}{{\em
  Computational Mathematics and Mathematical Physics} {\bfseries 57} no.~10,
  (Oct, 2017) 1635--1649}. \url{https://doi.org/10.1134/S0965542517100050}.

\bibitem{Bukzhalev_2017_MUCMC}
E.~E. Bukzhalev, ``The {C}auchy problem for singularly perturbed weakly
  nonlinear second-order differential equations: An iterative method,''
  \href{http://dx.doi.org/10.3103/S0278641917030037}{{\em Moscow University
  Computational Mathematics and Cybernetics} {\bfseries 41} no.~3, (Jul, 2017)
  113--121}. \url{https://doi.org/10.3103/S0278641917030037}.

\bibitem{2017arXiv171000640B}
E.~E. {Bukzhalev}, ``{On the Global Continuity of the Roots of Families of
  Monic Polynomials (in Russian)},'' {\em ArXiv e-prints} (Sept., 2017) ,
  \href{http://arxiv.org/abs/1710.00640}{{\ttfamily arXiv:1710.00640
  [math.CA]}}.

\bibitem{Ostrowski_1966_book}
A.~M. Ostrowski, {\em Solution of Equations and Systems of Equations}.
\newblock Pure and Applied Mathematics: A Series of Monographs and Textbooks,
  Vol. 9. Academic Press, New York and London, 2nd~ed., 1966.

\bibitem{Markushevich_2005_book}
A.~I. Markushevich, {\em Theory of Functions of a Complex Variable}.
\newblock AMS Chelsea Publishing Series, Vol. 296. AMS Chelsea Publishing,
  Providence, RI, 2nd~ed., 2005.

\bibitem{Gantmacher_2000_book}
F.~R. Gantmacher, {\em The Theory of Matrices, Vol. 2}.
\newblock AMS Chelsea Publishing Series, Vol. 133. AMS Chelsea Publishing,
  Providence, RI, 2000.

\end{thebibliography}\endgroup

\end{document}


\title{Равномерная экспоненциально"=степенная оценка решения семейства задач Коши\\ для линейных дифференциальных уравнений}

\author{Букжалёв Е. Е.\thanks{E-mail: bukzhalev@mail.ru}\\ \em \small Россия, Москва, Московский государственный университет имени М.~В.~Ломоносова} 

\maketitle

\begin{abstract}

Рассмотрено решение семейства задач Коши для линейных дифференциальных уравнений $m$-го порядка с постоянными коэффициентами, при этом в качестве параметров семейства выступают коэффициенты дифференциального уравнения и начальные значения решения и его производных до $(m-1)$-го порядка (под решением семейства задач понимается функция параметров данного семейства, которая каждому набору параметров ставит в соответствие одно из решений задачи с этими параметрами). Получена равномерная по параметрам (принадлежащим ограниченному множеству) экспоненциально-степенная оценка функций этого семейства. Кроме того, доказана непрерывность (по коэффициентам полинома) максимального элемента множества действительных частей корней приведённого полинома (произвольной степени $m$), использующаяся при установлении указанной оценки (поскольку каждому набору коэффициентов дифференциального уравнения отвечает свой характеристический полином с этими коэффициентами, то множество корней характеристического полинома, а вместе с ним и максимальный элемент множества действительных частей этих корней, также являются функцией коэффициентов).

\textbf{Ключевые слова:} семейства линейных дифференциальных уравнений, оценки решений дифференциальных уравнений, начальная задача для обыкновенного дифференциального уравнения, оценки корней многочленов, критерий устойчивости Гурвица.

\end{abstract}

\section{Введение}

Настоящая работа посвящена установлению равномерной (по параметрам $M_m \in \CC^m$ и $N_m \in \CC^m$) экспоненциально"=степенной оценки решения $w(\cdot; M_m, N_m): [0, +\infty) \to \CC$ задачи Коши для линейного дифференциального уравнения произвольного порядка $m \in \NN$ с постоянными коэффициентами $M_m$, рассматриваемыми как параметры для $w$ (при этом в качестве параметров выступают также начальные значения $N_m$ для $w(\cdot; M_m, N_m)$ и его производных):
\begin{gather}\label{de wn}
  w^{(m)}(\xi; M_m, N_m) = a_{m-1}\, w^{(m-1)}(\xi; M_m, N_m) + \cdots + a_0\, w(\xi; M_m, N_m), \quad \xi \in (0,+\infty);
\\\label{ic wn}
  w(0;M_m,N_m) = w^0,\ \ldots,\ w^{(m-1)}(0;M_m,N_m) = w^{m-1},
\end{gather}
где $M_m = (a_0, \dots, a_{m-1}) \in \CC^m$, $N_m = (w^0, \dots, w^{m-1}) \in \CC^m$, $w^{(i)}$ --- производная $i$-го порядка функции $w(\cdot; M_m, N_m)$ (то есть производная $i$-го порядка функции $w$ по её первому аргументу). Вместе с самим решением $w(\cdot; M_m, N_m)$ мы одновременно оценим его производные до $(m-1)$-го порядка (напомним, что задача Коши для уравнения \eqref{de wn} имеет единственное решение при любых значениях $a_i$ и $w^i$). Но поскольку в силу уравнения \eqref{de wn} производная $w(\cdot; M_m, N_m)$ порядка $n \ge m$ является линейной комбинацией младших производных, то равномерные экспоненциально-степенные оценки, на самом деле, справедливы для производных произвольного порядка.

\begin{Def}
  Пусть $\mathfrak K$ --- отображение, которое каждой двойке $(M_m,N_m) \in {(\CC^m)}^2 \hm=: D_{\mathfrak K}$ ставит в соответствие задачу \eqref{de wn}--\eqref{ic wn}. Отображение $\mathfrak K$ будем называть семейством задач Коши для дифференциального уравнения $m$-го порядка \textup(с постоянными коэффициентами\textup), а компоненты $M_m$ и $N_m$ двойки $(M_m,N_m) \in D_{\mathfrak K}$ --- параметрами этого семейства.
\end{Def}

\begin{rem}
  Сужение отображения $\mathfrak K$ на некоторое множество $\mathcal M \subseteq D_{\mathfrak K}$ часто называют подсемейством семейства $\mathfrak K$ и обозначают $\{ \mathfrak K(M_m, N_m) \}_{(M_m,N_m) \in \mathcal M}$.
\end{rem}

\begin{Def}
  Пусть $W$ --- отображение, которое каждой двойке $(M_m,N_m) \in D_{\mathfrak K}$ ставит в соответствие решение задачи $\mathfrak K(M_m,N_m)$ (то есть задачи \eqref{de wn}--\eqref{ic wn}). Отображение $W$ будем называть решением семейства $\mathfrak K$.
\end{Def}

\begin{rem}
  Частное значение $W(M_m,N_m)$ отображения $W$ представляет собой функцию одной вещественной переменной --- решение задачи \eqref{de wn}--\eqref{ic wn}. Но поскольку ранее мы уже обозначили решение этой задачи через $w(\cdot; M_m, N_m)$, то между $W$ и $w$ имеется следующая связь: $W(M_m,N_m)(\xi) = w(\xi; M_m, N_m)$ (строго говоря, это соотношение и определяет $w$).
\end{rem}

Пусть $\mathfrak E$ --- отображение, которое каждому $M_m \in \CC^m$ ставит в соответствие дифференциальное уравнение семейства $\mathfrak K(M_m,\cdot) := \{ \mathfrak K(M,N_m) \}_{(M,N_m) \in \{M_m\} \ts \CC^m}$ (см. \eqref{de wn}). Известно (см., например, \cite{tikh_vas_svesh_98_book}), что структура решения $W(M_m,N_m)$ задачи $\mathfrak K(M_m,N_m)$ зависит от распределения кратностей корней характеристического многочлена $\mathfrak P(M_m)$ (см. \eqref{PM}) уравнения $\mathfrak E(M_m)$. В связи с этим использование явных выражений для $w^{(i)}(\xi; M_m, N_m)$ с целью установления равномерных по $M_m$ и $N_m$ оценок функций $w^{(i)}(\cdot; M_m, N_m)$ проблематично (исключение составляют случаи $m=1$ и $m=2$ --- см. работы \cite{Bukzhalev_2017_ZhVM, Bukzhalev_2017_VMUSer15}, в которых указанные оценки получены для случая дифференциального уравнения второго порядка ($m=2$) с постоянными вещественными коэффициентами).

Далее нам понадобится равномерная (по коэффициентам $a_i$) оценка $\bLm$ --- максимальной из действительных частей корней характеристического многочлена уравнения $\mathfrak E(M_m)$. Её установление будет опираться на непрерывность $\bLm$ (как функции $M_m$), с доказательства которой мы и начнём нашу статью.

\section{Доказательство непрерывности \boldmath \texorpdfstring{$\bLm$}
{\unichar{"035E}\unichar{"1D27}\unichar{"1D50}}}

Пусть $\mathfrak P$ --- отображение, которое каждому $M_m = (a_0, \dots, a_{m-1}) \in \CC^m$ ставит в соответствие характеристический многочлен уравнения $\mathfrak E(M_m)$ (см. \eqref{de wn}):
\begin{gather}\label{PM}
  \mathfrak P(M_m) := \la^m - a_{m-1}\, \la^{m-1} - \cdots - a_1\, \la - a_0.
\end{gather}
Отображение $\mathfrak P$ будем называть также семейством многочленов. Так как степень многочлена $\mathfrak P(M_m)$ равна $m$ при всех $M_m \in \CC^m$, то найдутся такие функции $\la^1$, \dots, $\la^m$: $\CC^m \to \CC$, что
\begin{gather}\label{PME}
  \forall M_m \in \CC^m\ \mathfrak P(M_m) = (\la - \la^1(M_m)) \ldots (\la - \la^m(M_m)).
\end{gather}
При этом $\la^1(M_m)$, \dots, $\la^m(M_m)$ называют корнями многочлена $\mathfrak P(M_m)$. Кортеж $(\la^1, \dots,$ $\la^m)$ функций $\la^i$: $\CC^m \to \CC$, удовлетворяющих условию \eqref{PME}, назовём полным набором корней отображения $\mathfrak P$. Поскольку при каждом $M_m \in \CC^m$ можно по"=разному нумеровать корни многочлена $\mathfrak P(M_m)$, то полных наборов корней бесконечно много. Далее под $\la^1$, \dots, $\la^m$ будем подразумевать компоненты одного и того же (выбранного произвольным образом) полного набора корней отображения $\mathfrak P$, а под непрерывностью (на множестве $\mathds M$) кортежа $(\la^1, \dots, \la^m)$ будем подразумевать непрерывность (на $\mathds M$) каждой его компоненты.

Обозначим через $\mathcal L$ множество всех отображений $\CC^m \to \CC^m$, $M \mapsto (\la^1(M), \dots, \la^m(M))$, таких что $(\la^1, \dots, \la^m)$ --- полный набор корней отображения $\mathfrak P$. Выбор $\La \hm\in \mathcal L$ по сути означает задание нумерации корней многочлена $\mathfrak P(M_m)$ для всякого $M_m \in \CC^m$ (при этом каждый корень многочлена $\mathfrak P(M_m)$ повторяется столько раз, какова его кратность). Несложно убедиться в том, что при $m \ge 2$ множество $\mathcal L$ не содержит ни одного отображения, непрерывного во всём пространстве $\CC^m$ (более того, при $m \ge 4$ множество $\mathcal L$ не содержит ни одного отображения, сужение которого на $\RR^m$ было бы непрерывным,~--- доказательство см. в \cite{2017arXiv171000640B}). Известно, однако (см., например, \cite{Ostrowski_63_book}), что при любом $m$ для всякой точки $M_0 \in \CC^m$ найдётся отображение $\La_{M_0} \in \mathcal L$, непрерывное в этой точке.

\begin{rem}
  Каждое отображение $\La_{M_0}$ непрерывно в своей точке $M_0$. То есть мы заранее фиксируем точку и под неё подбираем отображение, непрерывное в ней. При этом оно может быть разрывно в других точках. Мы можем сменить точку, но тогда, вообще говоря, придётся сменить и отображение. Всякая точка $M_0$ имеет своё (непустое) множество $\mathfrak S_{M_0}$ полных наборов корней отображения $\mathfrak P$, непрерывных в этой точке, и пересечение семейства множеств $\{ \mathfrak S_{M_0} \}_{M_0 \in \CC^m}$ при $m \ge 2$ не содержит ни одного кортежа.
\end{rem}

\begin{lem}\label{lem}
  Пусть $\mathcal L \ni \La: M \mapsto (\la^1(M), \dots, \la^m(M))$. Тогда функция
  \begin{gather}\label{La}
    \bLm: \CC^m \to \RR,\ M_m \mapsto \max \{\re\la^1(M_m), \dots, \re\la^m(M_m)\}
  \end{gather}
  непрерывна на $\CC^m$.
\end{lem}

\begin{rem}
  Для всякой точки $M_m \in \CC^m$ неупорядоченная совокупность корней многочлена $\mathfrak P(M_m)$, а вместе с ней и величина $\bLm(M_m)$, не зависят от выбора $\La \in \mathcal L$. Таким образом, каждому $\La \in \mathcal L$ (то есть каждому способу нумерации корней многочленов семейства $\mathfrak P$) отвечает одна и та же функция $\bLm$.
\end{rem}

\begin{proof}[Доказательство леммы \ref{lem}]
Зафиксируем произвольную точку $M_0 \in \CC^m$ и выберем в $\mathcal L$ отображение $\La_{M_0}: M \mapsto (\la^1_{M_0}(M)$, $\dots,$ $\la^m_{M_0}(M))$, непрерывное в этой точке. При этом каждая из функций $\la^i_{M_0}$ также непрерывна в точке $M_0$. Но непрерывность $\la^i_{M_0}$ влечёт непрерывность $\re \la^i_{M_0}$, а из непрерывности всех $\re \la^i_{M_0}$, в свою очередь, вытекает непрерывность максимума этих функций.
\end{proof}

\section{Установление \texorpdfstring{экспоненциально"=степенных}
{экспоненциально-\hspace{0.001pt}степенных} оценок}

Обозначим $\Pi_m(C) := \{ (x_1, \dots, x_m) \in \CC^m : |x_1| \le C, \dots, |x_m| \le C \}$.

\begin{sta}\label{sta}

Пусть $C_a \ge 0$, $C_w \ge 0$. Тогда найдётся такое $\tl C_m \ge 0$, что
\begin{gather}\label{est w}
  \lt| w^{(i)}(\xi;M_m,N_m) \rt| \le \tl C_m\, (1 + \xi^{m-1})\, e^{\bLm(M_m)\, \xi}
\end{gather}
при всех $(i,\xi,M_m,N_m) \in \{ 0, \dots, m-1 \} \ts [0, +\infty) \ts \Pi_m(C_a) \ts \Pi_m(C_w)$, где $w(\cdot; M_m, N_m)$ --- решение задачи \eqref{de wn}--\eqref{ic wn}, $\bLm$ --- функция из леммы \ref{lem}.

\end{sta}

\begin{proof}

Воспользуемся методом математической индукции. Обозначим через $S_m$ высказывание утверждения \ref{sta}. Так как справедливость $S_1$ очевидна, то остаётся доказать, что для всякого натурального $m$ из $S_m$ следует $S_{m+1}$.

Рассмотрим задачу Коши для уравнения $(m+1)$-го порядка:
\begin{gather}\label{de wn+1}
  w^{(m+1)}(\xi; M_m, N_m) = a_m\, w^{(m)}(\xi; M_m, N_m) + \cdots + a_0\, w(\xi; M_m, N_m), \quad \xi \in (0,+\infty);
\\\label{ic wn+1}
  w(0;M_{m+1},N_{m+1}) = w^0,\ \ldots,\ w^{(m)}(0;M_{m+1},N_{m+1}) = w^m
\end{gather}
и зафиксируем произвольные неотрицательные $C_a$ и $C_w$. Утверждение $S_{m+1}$ состоит в том, что найдётся настолько большое $\tl C_{m+1}$, что
\begin{gather*}
  \lt| w^{(i)}(\xi;M_{m+1},N_{m+1}) \rt| \le \tl C_{m+1}\, (1 + \xi^m)\, e^{\bar \La_{m+1}(M_{m+1})\, \xi}
\end{gather*}
при всех $(i,\xi,M_{m+1},N_{m+1}) \in \ol{0,m} \ts [0, +\infty) \ts \Pi_{m+1}(C_a) \ts \Pi_{m+1}(C_w)$, где $w(\cdot; M_{m+1}, N_{m+1})$~--- решение задачи \eqref{de wn+1}--\eqref{ic wn+1}.

Чтобы установить справедливость $S_{m+1}$ (в предположении истинности высказывания $S_m$) сделаем замену зависимой переменной задачи \eqref{de wn+1}--\eqref{ic wn+1}:
\begin{gather}\label{u}
  w(\xi;M_{m+1},N_{m+1}) = e^{\la^*(M_{m+1})\, \xi}\, u(\xi;M_{m+1},N_{m+1}),
\end{gather}
где $\la^*$ --- функция, которая каждому $M_{m+1} = (a_0, \dots, a_m) \in \CC^{m+1}$ ставит в соответствие один из корней (любой) $\la_i(M_{m+1})$ характеристического многочлена уравнения \eqref{de wn+1}, действительная часть $\re \la_i(M_{m+1})$ которого совпадает с $\bar \La_{m+1}(M_{m+1})$:
\begin{gather}\label{re la*}
  \re \la^*(M_{m+1}) = \bar \La_{m+1}(M_{m+1}).
\end{gather}

Для новой функции $u(\cdot; M_{m+1}, N_{m+1}): [0, +\infty) \to \CC$ получается следующая начальная задача:
\begin{gather}\notag
  u^{(m+1)}(\xi;M_{m+1},N_{m+1}) = b_m(M_{m+1})\, u^{(m)}(\xi;M_{m+1},N_{m+1}) + \cdots +{}
\\\label{de u}
  {}+ b_1(M_{m+1})\, u'(\xi;M_{m+1},N_{m+1}), \quad \xi \in (0,+\infty);
\\\label{ic u}
  u(0;M_{m+1},N_{m+1}) = u^0(M_{m+1},N_{m+1}),\ \ldots,\ u^{(m)}(0;M_{m+1},N_{m+1}) = u^m(M_{m+1},N_{m+1}).
\end{gather}
Здесь $b_i(M_{m+1}) = \tl b_i(\la^*(M_{m+1}), M_{m+1})$, $u^i(M_{m+1},N_{m+1}) = \tl u^i(\la^*(M_{m+1}), N_{m+1})$, где, в свою очередь, $M_{m+1} = (a_0, \dots, a_m)$ и $N_{m+1} = (w^0, \dots, w^m)$, $\tl b_i$ и $\tl u^i$ --- известные функции $m+2$ аргументов (полиномиальные по первому аргументу и линейные остальным $m+1$ аргументам). Записывая уравнение \eqref{de u}, мы сразу же учли, что его характеристический многочлен при любом $M_{m+1} \in \CC^{m+1}$ имеет нулевой корень (см. \eqref{mu}), в связи с чем коэффициент $b_0(M_{m+1})$ при $u(\xi;M_{m+1},N_{m+1})$ есть тождественный нуль.

Ввиду \eqref{u}, при каждом $M_{m+1} \in \CC^{m+1}$ корнями характеристического многочлена уравнения \eqref{de u} будут:
\begin{gather}\label{mu}
  \mu_i(M_{m+1}) := \la_i(M_{m+1}) - \la^*(M_{m+1}), \quad i \in \{1, \dots, m+1\}.
\end{gather}
Отсюда и из определения $\la^*(M_{m+1})$ вытекает, что
\begin{gather}\label{Re mu}
  \re \mu_i(M_{m+1}) \le 0
\end{gather}
при всех $(i,M_{m+1}) \in \{1, \dots, m+1\} \ts \CC^{m+1}$.

Поскольку мы считаем, что точки $M_{m+1} = (a_0, \dots, a_m)$ принадлежат конечному параллелепипеду $P_{m+1}(C_a)$, то для всех корней $\la_i(M_{m+1})$ характеристического многочлена уравнения \eqref{de wn+1} справедливо (см., например, \cite{mar_68_book}):
\begin{gather}\label{la}
  |\la_i(M_{m+1})| \le 1 + C_a.
\end{gather}
Но тогда заведомо найдутся такие неотрицательные постоянные $C_b$ и $C_u$, что
\begin{gather}\label{bi ui}
  |b_i(M_{m+1})| \le C_b, \quad |u^i(M_{m+1},N_{m+1})| \le C_u
\end{gather}
при всех $(i,M_{m+1},N_{m+1}) \in \ol{0,m} \ts \Pi_{m+1}(C_a) \ts \Pi_{m+1}(C_w)$.

Понизим порядок уравнения \eqref{de u}, сделав ещё одну замену зависимой переменной:
\begin{gather}\label{v}
  u'(\xi;M_{m+1},N_{m+1}) = v(\xi;M_{m+1},N_{m+1}).
\end{gather}

Функция $v(\cdot; M_{m+1}, N_{m+1}): [0, +\infty) \to \CC$ удовлетворяет следующей начальной задаче:
\begin{gather}\notag
  v^{(m)}(\xi;M_{m+1},N_{m+1}) = b_m(M_{m+1})\, v^{(m-1)}(\xi;M_{m+1},N_{m+1}) + \cdots +{}
\\\label{de v}
  {}+ b_1(M_{m+1})\, v(\xi;M_{m+1},N_{m+1}), \quad \xi \in (0,+\infty);
\\\notag
  v(0;M_{m+1},N_{m+1}) = u^1(M_{m+1},N_{m+1}),\ \ldots,\ v^{(m-1)}(0;M_{m+1},N_{m+1}) = u^m(M_{m+1},N_{m+1}).
\end{gather}

Пусть $\nu_1(M_{m+1})$, \dots, $\nu_m(M_{m+1})$ --- корни характеристического многочлена уравнения \eqref{de v}. Поскольку каждое из $\nu_i(M_{m+1})$ одновременно является корнем характеристического многочлена уравнения \eqref{de u}, то для них, так же как и для $\mu_i(M_{m+1})$ (см. \eqref{Re mu}), при всех $M_{m+1} \in \CC^{m+1}$ справедливо:
\begin{gather}\label{Re nu}
  \re \nu_i(M_{m+1}) \le 0.
\end{gather}

Заметим ещё, что (см. \eqref{bi ui})
\begin{gather*}
  M_m = (b_1(M_{m+1}), \dots, b_m(M_{m+1})) \in \Pi_m(C_b),
\\
  N_m = (u^1(M_{m+1},N_{m+1}), \dots, u^m(M_{m+1},N_{m+1})) \in \Pi_m(C_u)
\end{gather*}
при всех $M_{m+1} \in \Pi_{m+1}(C_a)$ и $N_{m+1} \in \Pi_{m+1}(C_w)$. Последние оценки открывают возможность для применения предположения индукции к функции $v$, согласно которому найдётся такое $\tl C_m \ge 0$, что (см. \eqref{est w}, \eqref{La} и \eqref{Re nu})
\begin{gather}\label{est v}
  \lt| v^{(i)}(\xi;M_{m+1},N_{m+1}) \rt| \le \tl C_m\, (1 + \xi^{m-1})
\end{gather}
при всех $(i,\xi,M_{m+1},N_{m+1}) \in \{ 0, \dots, m-1 \} \ts [0, +\infty) \ts \Pi_{m+1}(C_a) \ts \Pi_{m+1}(C_w)$.

Из \eqref{v} и \eqref{est v} для первых $m$ производных функции $u(\cdot; M_{m+1}, N_{m+1})$ сразу же имеем (здесь и до конца доказательства подразумевается, что $(\xi,M_{m+1},N_{m+1}) \in [0, +\infty) \hm\ts \Pi_{m+1}(C_a) \ts \Pi_{m+1}(C_w)$):
\begin{gather}\label{est du}
  \lt| u^{(i)}(\xi;M_{m+1},N_{m+1}) \rt| = \lt| v^{(i-1)}(\xi;M_{m+1},N_{m+1}) \rt| \le \tl C_m\, (1 + \xi^{m-1}).
\end{gather}

Чтобы оценить саму функцию $u$, проинтегрируем \eqref{v}, после чего воспользуемся \eqref{ic u}, \eqref{bi ui} и \eqref{est v}, а также свойством монотонности и оценкой модуля определённого интеграла:
\begin{gather}\notag
  \Big| u(\xi;M_{m+1},N_{m+1}) \Big| = \Big| u(0;M_{m+1},N_{m+1}) + \tintl0\xi\, v(\zt;M_{m+1},N_{m+1})\, d\zt \Big| \le \Big| u^0(M_{m+1},N_{m+1}) \Big| +{}
\\\label{est u}
  {}+ \tintl0\xi\, \Big| v(\zt;M_{m+1},N_{m+1}) \Big|\, d\zt \le C_u + \tintl0\xi\, \tl C_m\, (1 + \xi^{m-1})\, d\zt \le \tl C_u\, (1 + \xi^m)
\end{gather}
при достаточно большом $\tl C_u$.

Вернёмся к $w$. Из \eqref{u}, \eqref{est u}, \eqref{est du}, \eqref{la}, \eqref{re la*} и формулы Лейбница (для производной $i$-го порядка произведения двух функций) при каждом $i \in \ol{0,m}$ и достаточно большом $\tl C_{m+1}$ имеем:
\begin{gather*}
  \lt| w^{(i)}(\xi;M_{m+1},N_{m+1}) \rt| \le \sum_{j=0}^i \frac{i!}{j!(i-j)!}\, \big| u^{(j)}(\xi;M_{m+1},N_{m+1}) \big|\, {\big| \la^*(M_{m+1}) \big|}^{i-j}\, \big| e^{\la^*(M_{m+1})\, \xi} \big| \le
\\
  \le \Big[ \tl C_u\, (1 + \xi^m)\, {(1 + C_a)}^i + \sum_{j=1}^i \frac{i!}{j!(i-j)!}\, \tl C_m\, (1 + \xi^{m-1})\, {(1 + C_a)}^{i-j} \Big]\, e^{\re \la^*(M_{m+1})\, \xi} \le
\\
  \le \tl C_{m+1}\, (1 + \xi^m)\, e^{\bar \La_{m+1}(M_{m+1})\, \xi}.
\end{gather*}

\end{proof}

\begin{rem}
  Поскольку в силу уравнения \eqref{de wn} производная функции $w(\cdot; M_m, N_m)$ порядка $i \ge m$ является линейной комбинацией её младших производных, то с помощью индукции по $i$ несложно убедиться в том, что для любого целого неотрицательного $i$ найдётся такое $\tl C^i_m \ge 0$, что
  \begin{gather*}
    \lt| w^{(i)}(\xi;M_m,N_m) \rt| \le \tl C^i_m\, (1 + \xi^{m-1})\, e^{\bLm(M_m)\, \xi}
  \end{gather*}
  при всех $(\xi,M_m,N_m) \in [0, +\infty) \ts \Pi_m(C_a) \ts \Pi_m(C_w)$.
\end{rem}

\begin{rem}
  Если в утверждении \ref{sta} вместо прямоугольников $\Pi_m(C_a)$ и $\Pi_m(C_w)$ взять произвольные ограниченные множества $\mathds D_a$ и $\mathds D_w$ (в том числе, состоящие только из действительных чисел), оно, очевидно, останется справедливым.
\end{rem}

\begin{rem}
  Под равномерностью оценки \eqref{est w} понимается независимость коэффициента $\tl C_m$ от параметров $a_i$ и $w^i$. В то же время коэффициент в аргументе экспоненты от этих параметров зависит. Таким образом, оценка \eqref{est w} является, в некотором смысле, лишь наполовину равномерной. Можно, конечно, воспользоваться \eqref{la}, и заменить в \eqref{est w} $\bLm(M_m)$ на $1 + C_a$, получив полностью равномерную, но вместе с тем, вообще говоря, более грубую оценку. Однако иногда настолько сильное огрубление бывает нежелательным, особенно если при этом происходит смена знака (с отрицательного на положительный) коэффициента при $\xi$ в аргументе экспоненты (как раз о таком случае пойдёт речь ниже).
\end{rem}

Пусть $\mathds D_a$ --- замкнутое ограниченное множество и пусть известно (например, в силу критерия Гурвица --- см. \cite{Demidovich_67_book}), что $\forall (i,M_m) \in  \ol{1,m} \ts \mathds D_a$ $\re\la^i(M_m) < 0$. Тогда $\forall M_m \in \mathds D_a$ $\bLm(M_m) < 0$, и так как согласно лемме \ref{lem} функция $\bLm$ непрерывна на всём $\CC^m$ (а значит, и на любом множестве $\mathds D_a$ этого пространства), то по теореме Вейерштрасса об экстремальных значениях непрерывной функции $\exists M^*_m \in \mathds D_a$ $\forall M_m \in \mathds D_a$ $\bLm(M_m) \le \bLm(M^*_m) < 0$. Отсюда и из утверждения \ref{sta} вытекает существование такого $\tl C_m > 0$, что для решения $w(\cdot; M_m, N_m)$ задачи \eqref{de wn}--\eqref{ic wn} при всех $(i,\xi,M_m,N_m) \in \{ 0, \dots, m-1 \} \ts [0, +\infty) \ts \mathds D_a \ts \mathds D_w$ (здесь $\mathds D_w$ --- произвольное ограниченное подмножество $\CC^m$) справедливо:
\begin{gather}\label{est_w}
  \lt| w^{(i)}(\xi;M_m,N_m) \rt| \le \tl C_m\, (1 + \xi^{m-1})\, e^{- \varkappa\, \xi},
\end{gather}
где $\varkappa := - \bLm(M^*_m) > 0$.

В заключение рассмотрим семейство (по параметрам, которые перечисляются ниже) задач Коши для линейного дифференциального уравнения произвольного порядка $m \in \NN$ с коэффициентами, зависящими от параметров $t_1$, \dots, $t_k$ (кроме того, в качестве параметров по"=прежнему выступают начальные значения для решения $w(\cdot; M_k, N_m)$ и его производных):
\begin{gather}\notag
  w^{(m)}(\xi; M_k, N_m) = a_{m-1}(t_1, \dots, t_k)\, w^{(m-1)}(\xi; M_k, N_m) + \cdots +{}
\\\label{de_wnt}
  {}+ a_0(t_1, \dots, t_k)\, w(\xi; M_k, N_m) = 0, \quad \xi \in (0,+\infty);
\\\notag
  w(0;M_k,N_m) = w^0,\ \ldots,\ w^{(m-1)}(0;M_k,N_m) = w^{m-1},
\end{gather}
где $M_k = (t_1, \dots, t_k) \in \mathds D_t \subseteq \CC^k$, $N_m = (w^0, \dots, w^{m-1}) \in \CC^m$, $a_i: \mathds D_t \to \CC$.

\begin{sta}\label{sta_t}

Пусть $\mathds D_t$ --- замкнутое ограниченное подмножество $\CC^k$, $\mathds D_w$ --- ограниченное подмножество $\CC^m$, $a: \mathds D_t \ni M_k \mapsto (a_0(M_k), \dots, a_{m-1}(M_k)) \in \CC^m$, функции $a_0$, \dots, $a_{m-1}$ непрерывны на $\mathds D_t$, $\forall M_k \in \mathds D_t$ $\bLm(a(M_k)) < 0$ (здесь $\bLm$ --- функция из леммы \ref{lem}). Тогда найдутся такие $\tl C_m \ge 0$ и $\vkp > 0$, что
\begin{gather*}
  \lt| w^{(i)}(\xi;M_k,N_m) \rt| \le \tl C_m\, (1 + \xi^{m-1})\, e^{ - \vkp\, \xi}
\end{gather*}
при всех $(i,\xi,M_k,N_m) \in \{ 0, \dots, m-1 \} \ts [0, +\infty) \ts \mathds D_t \ts \mathds D_w$, где $w(\cdot; M_k, N_m)$ --- решение задачи \eqref{de_wnt}.

\end{sta}

\begin{proof}
Поскольку $a(\mathds D_t) =: \mathds D_a$ --- ограниченное замкнутое подмножество $\CC^m$, то для доказательства справедливости утверждения \ref{sta_t} достаточно заметить, что
\[
  \max\limits_{M_k \in \mathds D_t} \bLm(a(M_k)) = \max\limits_{M_m \in \mathds D_a} \bLm(M_m),
\]
и, положив $\vkp := - \max\limits_{M_k \in \mathds D_t} \bLm(a(M_k))$, воспользоваться оценкой \eqref{est_w}.
\end{proof}